\documentclass[11pt,a4paper]{article}
\usepackage{amsmath}
\usepackage{amsthm}
\usepackage{makeidx}
\usepackage{latexsym}
\usepackage{graphicx}
\usepackage{amssymb}
\usepackage{hyperref}
\newtheorem{teo}{Theorem}
\newtheorem{lemma}{Lemma}
\newtheorem{prop}{Proposition}

\newtheorem{oss}{Remark}

\theoremstyle{remark}
\newtheorem{es}{\textbf{Example}}
\usepackage[latin1]{inputenc}
\title{Periodic representations and rational approximations of square roots}
\author{Marco Abrate, Stefano Barbero, Umberto Cerruti, Nadir Murru}
\date{}

\begin{document}
\maketitle

\begin{abstract}
In this paper the properties of R\'edei rational functions are used to derive rational approximations for square roots and both Newton and Pad\'e approximations are given as particular cases. Moreover, R\'edei rational functions are introduced as convergents of particular periodic continued fractions and are applied for approximating square roots in the field of $p$--adic numbers and to study periodic representations. Using the results over the real numbers, we show how to construct periodic continued fractions and approximations of square roots which are simultaneously valid in the real and in the $p$--adic field.
\end{abstract}

\section{Introduction}
Diophantine approximation is a very rich research field and it is actually very studied and developed. The research of rational approximations for irrational numbers can be performed in many different ways. Continued fractions are the most used objects in this context, since they have many important approximation properties (e.g., they provide the best approximations for irrational numbers). Recently, different kind of matrices has been used for finding approximations of irrational numbers. For example in \cite{Wild} some $2\times2$ matrices are used in order to generate an infinite number of solutions of the Pell equation and in this way we have infinite approximations of square roots. Applying matrix powers techniques in this context is very useful. Since the entries of a power matrix recur with the characteristic polynomial of the starting matrix, we get another important tools in this subject, based on the deep theory about linear recurrent sequences. In \cite{Rosen} the powers of matrices
\begin{equation} \label{matrices} \begin{pmatrix} z & d \cr 1 & z  \end{pmatrix}\quad\text{and} \quad \begin{pmatrix} z+1 & d \cr 1 & z+1  \end{pmatrix}\ , \end{equation}
yield to approximations for $\sqrt{d}$, where $z=\lfloor \sqrt{d} \rfloor$. These approximations are related to continued fractions and minus continued fractions. In this paper we will see that they coincide with the R\'edei rational functions \cite{Redei}. R\'edei rational functions (see \cite{Lidl} for a good survey) arise from the expansion of $(z+\sqrt{d})^n$, where $z$ is an integer and $d$ is a nonsquare positive integer. The explicit expression for this expansion is 
\begin{equation}\label{pow}(z+\sqrt{d})^n=N_n(d,z)+D_n(d,z)\sqrt{d}\ ,\end{equation}
where
$$ N_n(d,z)=\sum_{i=0}^{[n/2]}\binom{n}{2i}d^iz^{n-2i}\quad \text{and} \quad D_n(d,z)=\sum_{i=0}^{[n/2]}\binom{n}{2i+1}d^iz^{n-2i-1}.  $$
The R\'edei rational functions $Q_n(d,z)$ are defined by 
\begin{equation}\label{red}Q_n(d,z)=\cfrac{N_n(d,z)}{D_n(d,z)}, \quad \forall n\geq1 \ .\end{equation}
R\'edei rational functions are very useful in many aspects of number theory. Some of their application are concerned with diophantine approximations, public key cryptographic system \cite{Nob} and generation of pseudorandom sequences \cite{Topu}. Furthermore, given a finite field $\mathbb F_q$, of order $q$, and $\sqrt{d}\not\in\mathbb F_q$, then $Q_n(d,z)$ is a permutation of $\mathbb F_q$ if and only if $(n,q+1)=1$ (see \cite{Lidl}, p. 44). Moreover, they provide approximations for square roots and they have some connections with continued fractions: it is straightforward to see that
$$\lim_{n\rightarrow\infty}Q_n(d,z)=\sqrt{d}, \quad \forall d,z\in\mathbb Z,$$
$d$ positive, not square. In \cite{Pell} the authors found a value for $d$ such that R\'edei rational functions coincide with the convergents of the continued fraction of $\sqrt{d}$ leading to the solutions of the Pell equation. Furthermore these functions have been generalized in order to study them over a general class of conics and develop rational approximations of irrational numbers over conics, obtaining a new result for quadratic irrationalities approximations \cite{Pellhd}. Now, we show how R\'edei rational functions are related with the approximations studied in \cite{Rosen}. We recall their matricial representation (see \cite{Von}).
\begin{prop} \label{matrix-redei} For every $d,z \in \mathbb Z$, $d$ positive nonsquare:
$$ \begin{pmatrix} z & d \cr 1 & z  \end{pmatrix}^n=\begin{pmatrix} N_n & dD_n \cr D_n & N_n \end{pmatrix}. $$
\end{prop}
The matrix used in the previous Proposition coincides with the matrices (\ref{matrices}) when $z=\lfloor\sqrt{d}\rfloor$ and $z=\lceil \sqrt{d} \rceil$ respectively, but we can observe that the matrix 
$$\begin{pmatrix} z & d \cr 1 & z  \end{pmatrix}$$
provides approximations of $\sqrt{d}$ for every choice of the integer $z$. \\
The previous proposition yields a recurrence relation for the R\'edei polynomials, because the entries of a matrix power recur with the characteristic polynomial of the matrix. In this case the starting matrix has trace $2z$ and determinant $z^2-d$ and we have \begin{equation}\label{recred}
\begin{cases}(N_n(d,z))_{n=0}^{+\infty}=\mathcal{W}(1,z,2z,z^2-d)\cr
(D_n(d,z))_{n=0}^{+\infty}=\mathcal{W}(0,1,2z,z^2-d)\ .\end{cases}
\end{equation}
We indicate with $(c_n)_{n=0}^{+\infty}=\mathcal{W}(a,b,h,k)$ the linear recurrent sequence of order 2 with initial conditions $a$, $b$ and characteristic polynomial $t^2-ht+k$, i.e.,
$$\begin{cases} c_0=a \cr c_1=b \cr c_n=hc_{n-1}-kc_{n-2}, \quad \forall n\geq2 \ .   \end{cases}$$
In the next sections we deal with the R\'edei rational functions and we point out how they provide rational approximations for square roots. We will show that both Newton and Pad\'e approximations can be derived as particular cases. Moreover, R\'edei rational functions will be introduced in a totally new way as convergents of particular periodic continued fractions. Afterwards the study of approximations of irrationalities over the field of $p$--adic numbers is also considered. Many attempts of generalizations of continued fractions over the $p$--adic numbers have been performed, starting from Mahler \cite{Mal}. In \cite{Bro}, \cite{Lao}, and \cite{Moore} several algorithms which generalize the continued fractions over the $p$--adic numbers and a complete bibliography of the argument are showed. However, no algorithm has been found such that it always produces a periodic representation for every square root in the field of $p$--adic numbers. In the last section of this paper we use R\'edei rational functions for approximating square roots in the field of $p$--adic numbers and we study periodic representations. Using the results over the real numbers, we see how it is possible to construct periodic continued fractions and approximations of square roots which are simultaneously valid in the real and in the $p$--adic field.

\section{R\'edei rational functions and continued fractions}
A continued fraction is a representation of a real number $\alpha$ through a sequence of integers as follows:
$$\alpha=a_0+\cfrac{1}{a_1+\cfrac{1}{a_2+\cfrac{1}{a_3+\cdots}}}\ ,$$
where the integers $a_0,a_1,...$ can be evaluated with the recurrence relations 
$$\begin{cases} a_k=[\alpha_k]\cr \alpha_{k+1}=\cfrac{1}{\alpha_k-a_k} \quad \text{if} \ \alpha_k \ \text{is not an integer}  \end{cases}  \quad k=0,1,2,...$$
for $\alpha_0=\alpha$ (cf. \cite{Olds}). A continued fraction can be expressed in a compact way using the notation $[a_0,a_1,a_2,a_3,...]$. The finite continued fraction
$$[a_0,...,a_n]=\cfrac{p_n}{q_n}\ ,\quad n=0,1,2,...$$
is a rational number and is called the $n$--th \emph{convergent} of $[a_0,a_1,a_2,a_3,...]$. An important property of continued fractions involves quadratic irrationalities. A continued fraction is periodic if and only if it represents a quadratic irrationality. However, the period of such continued fractions can be very long and it is not possible to predict its length.\\
In this section we focus on a particular continued fraction with rational partial quotients which ever represents a square root. Using relations (\ref{recred}) we can easily prove that the R\'edei rational functions correspond to the convergents of the continued fractions
\begin{equation} \label{fcredei}  \sqrt{d}=\left[\ z,\overline{\cfrac{2z}{d-z^2},2z}\ \right].  \end{equation}
\begin{lemma} \label{fcr-conv} Let $\left[\cfrac{a_0}{b_0}\ ,\cfrac{a_1}{b_1}\ ,...,\cfrac{a_i}{b_i}\ ,...\right]$ be a continued fraction, $a_i, b_i \in \mathbb Z$ for $i=0,1,...$, and let $(p_n)_{n=0}^{+\infty},(q_n)_{n=0}^{+\infty}$ be the sequences of numerators and denominators of the convergents. Let us consider the sequences $(s_n)_{n=0}^{+\infty},(t_n)_{n=0}^{+\infty},(u_n)_{n=0}^{+\infty}$ defined by
\[  \begin{cases}  s_n=a_ns_{n-1}+b_nb_{n-1}s_{n-2} \cr t_n=a_nt_{n-1}+b_nb_{n-1}t_{n-2} \cr u_n=b_nu_{n-1}\ , \end{cases} \]
for every $n\geq2$, with initial conditions 
$$\begin{cases} s_0=a_0,s_1=a_0a_1+b_0b_1 \cr t_0=1,t_1=a_1 \cr u_0=1\ .\end{cases}$$
Then we have $p_n=\cfrac{s_n}{b_0u_n}$ and $q_n=\cfrac{t_n}{u_n}\ ,$ for every $n\geq0$. 
\end{lemma}
\begin{proof}
We prove the theorem by induction. We can directly verify the inductive basis. For $n=0$ and $n=1$ we have $p_0=\cfrac{a_0}{b_0}\ ,$ $q_0=1$ and $p_1=\cfrac{a_0}{b_0}\cdot\cfrac{a_1}{b_1}+1=\cfrac{a_0a_1+b_0b_1}{b_0b_1}\ ,$ $q_1=\cfrac{a_1}{b_1}\ ,$ so $p_0=\cfrac{s_0}{b_0u_0}\ ,$ $q_0=\cfrac{t_0}{u_0}$ and $p_1=\cfrac{s_1}{b_0u_1}\ ,$  $q_1=\cfrac{t_1}{u_1}\ .$ Finally, for $n=2$, it is easy to see that
$$p_2=\cfrac{a_0a_1a_2+a_2b_0b_1+a_0b_1b_2}{b_0b_1b_2}=\cfrac{a_2s_1+b_1b_2s_0}{b_0b_2u_1}=\cfrac{s_2}{b_0u_2}\ ,$$
and $$q_2=\cfrac{a_1a_2+b_1b_2}{b_1b_2}=\cfrac{a_2t_1+b_2b_1t_0}{b_1b_2}=\cfrac{t_2}{u_2}\ .$$
Now, if the thesis is true for any integer up to $n-1$, then for $n$ we have
$$p_n=\cfrac{a_n}{b_n}p_{n-1}+p_{n-2}=\cfrac{a_n}{b_n}\cdot\cfrac{s_{n-1}}{b_0u_{n-1}}+\cfrac{s_{n-2}}{b_0u_{n-2}}=\cfrac{a_ns_{n-1}u_{n-2}+b_nu_{n-1}s_{n-2}}{b_0b_nu_{n-1}u_{n-2}}=$$
$$=\cfrac{a_ns_{n-1}u_{n-2}+b_nb_{n-1}u_{n-2}s_{n-2}}{b_0u_nu_{n-2}}=\cfrac{u_{n-2}s_n}{b_0u_{n-2}u_n}=\cfrac{s_n}{b_0u_n}\ .$$
and similarly
$$q_n=\cfrac{a_n}{b_n}q_{n-1}+q_{n-2}=\cfrac{a_n}{b_n}\cdot\cfrac{t_{n-1}}{u_{n-1}}+\cfrac{t_{n-2}}{u_{n-2}}=\cfrac{a_nt_{n-1}u_{n-2}+b_nt_{n-2}u_{n-1}}{b_nu_{n-1}u_{n-2}}=$$
$$=\cfrac{a_nt_{n-1}u_{n-2}+b_{n-1}b_nt_{n-2}u_{n-2}}{u_{n-2}u_n}=\cfrac{t_n}{u_n}\ .$$
\end{proof}
\begin{oss}
Continued fractions with rational partial quotients have many interesting algebraic properties. In \cite{Acc}, the authors studied a 2--periodic continued fraction representing any quadratic irrationalities, showing that among its convergents there are at the same time Newton, Halley and secant approximations. In the following we will see that among the convergents of the continued fraction (\ref{fcredei}) we have at the same time Newton and Pad\`e approximations of $\sqrt{d}$.
\end{oss}
By the previous Lemma we find that the convergents of the continued fraction
$$\left[\ \overline{2z,\cfrac{2z}{d-z^2}}\ \right]$$
correspond to $\cfrac{\sigma_{n+2}}{\sigma_{n+1}}\ ,$ $n=0,1,2,\ldots$, where
$$(\sigma_n)_n=\mathcal W(0,1,2z,z^2-d)\ .$$
Thus, the convergents of (\ref{fcredei}) are equal to
$$\cfrac{\sigma_{n+2}}{\sigma_{n+1}}-z,\quad n=0,1,2,\ldots$$
and from the recurrence of the R\'edei polynomials we can observe that
$$\sigma_{n+1}-z\sigma_n=N_n(d,z), \quad \forall n\geq0$$
$$\sigma_n=D_n(d,z), \quad \forall n\geq0\ .$$
We summarize this result in the following
\begin{teo} \label{Q-conv}
Let $d$ be a positive integer not square, for every integer $z$ we have
$$\sqrt{d}=\left[\ z,\overline{\cfrac{2z}{d-z^2},2z}\ \right]$$
whose convergents are the R\'edei rational functions $Q_n(d,z)$, $\forall n\geq 1$.
\end{teo}
\begin{teo} \label{np}
Let $k$ be a positive integer not square, then
\begin{enumerate}
\item $Q_{2^n}(d,z)$ are the Newton approximations of $\sqrt{d}$ with initial condition $z$, for every $n\geq0$;
\item $Q_{2n+1}(d,z)$ are the Pad\`e approximations of $\sqrt{d}$ centered in $z^2$ and of degree $n$, for every $n\geq0$.
\end{enumerate}
\end{teo}
\begin{proof}
In general, the Newton method for approximating $\alpha$, real root of $f(x)=bx^2-ax-c$, provides a sequence of rationales $x_n$, by the equation 
\begin{equation*}  x_n=x_{n-1}-\cfrac{f(x_{n-1})}{f'(x_{n-1})}=x_{n-1}-\cfrac{bx^2_{n-1}-ax_{n-1}-c}{2bx_{n-1}-a}=\cfrac{bx^2_{n-1}+c}{2bx_{n-1}-a}\ , \end{equation*}
with a suitable initial condition $x_0$.
 We obtain the Newton iterator for $\sqrt{k}$ when $b=1$, $a=0$, $c=k$. The initial condition $x_0=d$ gives
$$\begin{cases} x_0=d \cr x_n=\cfrac{x_{n-1}^2+k}{2x_{n-1}} \ .\end{cases} $$
We have to point out that
$$Q_1(d,z)=\cfrac{N_1(d,z)}{D_1(d,z)}=z,\quad Q_2(d,Q_1(d,z))=Q_2(d,z)=\cfrac{N_2(d,z)}{D_2(d,z)}=\cfrac{z^2+d}{2z}\ ,$$ 
where we used the multiplicative property of $Q_n(d,z)$. Observing that $Q_2(d,\cdot)$ coincides with the Newton iterator, we have
$$Q_{2^n}(d,z)=x_n, \quad \forall n\geq0\ .$$
Furthermore, we have a similar result for the Pad\`e approximations. In this regard, we consider $Q_n(d,z)$ as a function only of the variable $d$ and we think to $z$ as a fixed integer. Remembering that
$$N_n(d,z)-D_n(d,z)\sqrt{d}=(z-\sqrt{d})^n$$
we have
$$f_n(d)= Q_n(d,z)-\sqrt{d}=\cfrac{(z-\sqrt{d})^n}{D_n(d,z)}\ . $$
When we differentiate $f_n(d)$ with respect to $d$, the previous equality allows us to easily observe that $f^{(i)}(z^2)=0$ for $i=1,2,\ldots,n-1$. Thus, in the Taylor series of $f_{2n+1}(d)$ centered in $d=z^2$ the first $2n$ terms are zero. As a direct consequence, considering the definition of Pad\`e approximation (see, e.g., \cite{Bak}), we have that $Q_{2n+1}(d,z)$ are the Pad\`e approximations of $\sqrt{d}$ centered in $z^2$, corresponding to the ratios of polynomials of degree $n$.
\end{proof}

This theorem has a really interesting consequence. Using this result we can evaluate Newton approximations of square roots by power matrices. In particular the $n$th approximation of $\sqrt{d}$ is given by the ratio of the entries of the first column of
$$\begin{pmatrix} z & d \cr 1 & z  \end{pmatrix}^{2^n}.$$
In this way we can evaluate the $n$th term of the Newton iteration without the evaluation of all the previous steps, but we can directly obtain it in a fast way. An analogue observation is valid for the Pad\'e approximations of $\sqrt{d}$.

\section{p--adic approximations of square roots}

We have seen that the functions $Q_n(d,z)$ approximate $\sqrt{d}$, for every integer $z$. Now, we see that the parameter $z$ has a precise role if we consider approximations of $\sqrt{d}$ in the field $\mathbb Q_p$ of the $p$--adic numbers, instead of $\mathbb R$.\\
We recall that given a prime number $p$, the $p$--adic numbers are objects of the form
$$a_mp^m+a_{m+1}p^{m+1}+a_{m+2}p^{m+2}+...$$
for $0\leq a_i\leq p-1$ integer, $m$ integer, and they form a field \cite{p-adic} with respect to the two obvious operations.\\

\begin{teo}\label{teo:p-adic}
Let $p$ be a prime number and $z$ an integer such that $z^2\equiv d \mod p$. Then $Q_n(d,z)$'s converge to $\sqrt{d}$ in $\mathbb Q_p$.
\end{teo}
\begin{proof} If we consider that the congruence
$$x^2\equiv d \mod p,$$
$p$ a prime, has solutions, then there exists $z$ such that 
$$z^2-d=np$$
for some integer $n$. Since
$$\begin{pmatrix}z & d \cr 1 & z\end{pmatrix}^n=\begin{pmatrix}  N_n & dD_n \cr D_n & N_n \end{pmatrix},$$
we have
$$N_n^2-dD_n^2\equiv 0 \mod p^n,\quad \forall n\geq1$$
or equivalently
$$\left(\cfrac{N_n}{D_n}\right)^2\equiv d\mod p^n,\quad \forall n\geq1,$$
i.e., $Q_n(d,z)$'s converge to $\sqrt{d}$ in the field of the $p$--adic numbers. This is the same as saying that $Q_n(d,z)$ are $p$--adic approximations of $\sqrt{d}$.
\end{proof}

By Theorem \ref{teo:p-adic} it follows that the choice of a convenient parameter $z$ is essential in the field of the $p$--adic numbers. Moreover, we can use R\'edei rational functions in order to obtain Newton approximations in a $p$--adic sense similarly to the real case. Indeed, let us consider $\sqrt{d}\in\mathbb Q_p$, i.e.,
$$\sqrt{d}=\sum_{i=0}^{\infty}b_ip^i,$$
for $b_i\in\{0,1,...,p-1\}$, and let us set $z=b_0$ such that $z^2\equiv d \mod p$. Setting 
$$a_n=\sum_{i=0}^nb_ip^i,\quad \forall n\geq0$$
we have $a_n^2\equiv d\mod p^{n+1}$ and
$$a_n=a_{n-1}+b_np^n.$$
Since
$$a_n^2=a_{n-1}^2+b_n^2p^{2n}+2a_{n-1}b_np^n\equiv a_{n-1}^2+2a_{n-1}b_np^n \mod p^{n+1}$$
we finally have
$$a_n\equiv \cfrac{a_{n-1}^2+a_n^2}{2a_{n-1}}\mod p^{n+1}\equiv \cfrac{a_{n-1}^2+d}{2a_{n-1}}\mod p^{n+1}$$
and recalling Theorem \ref{np} we have
$$a_n\equiv Q_{2^n}(d,z) \mod p^{n+1}$$
which coincide with the Newton approximations over the field of $p$--adic numbers of $\sqrt{d}$. Furthermore, we can observe 

$$
a_n\equiv Q_{n+1}(d,z) \mod p^{n+1}
$$

since $p^{n+1}$ divides $N^2_{n+1}(d,z)-dD^2_{n+1}(d,z)$. These observations about the role of R\'edei rational functions in the field of the $p$--adic numbers allow us to conclude that we can use periodic continued fraction

\begin{equation} \label{cfp} 
\left[\ z,\overline{\cfrac{2z}{d-z^2},2z}\ \right]  
\end{equation}

in order to give periodic representations of square roots in $\mathbb Q_p$. These continued fractions represent square roots in $\mathbb Q_p$ though they are not provided by a specific algorithm. However, it is interesting to observe that for some particular case the continued fraction (\ref{cfp}) coincides with the continued fraction obtained from an algorithm presented in \cite{Moore}.
\begin{es}
Let us consider $\sqrt{26}\in\mathbb Q_{229}$. It is possible to check that $x^2\equiv 26\mod {229}$ has $22$ as solution. So we take $z=22$ in (\ref{cfp}) and we obtain
$$
\sqrt{26}=\left[\ 22,\overline{-\cfrac{22}{229},44}\ \right]$$
which coincide with the expansion provided in \cite{Moore} (p. 17). Here it follows immediately that the continued fraction converge in a real sense too.
\end{es}
Finally, it is interesting to observe that when $z$ is such that $z^2\equiv d \mod p$ the continued fraction (\ref{cfp}) converges to $\sqrt{d}$ both in a real and in a $p$--adic sense and R\'edei rational functions $Q_n(d,z)$ provide simultaneously real and $p$--adic approximations of $\sqrt{d}$.

\end{document}